\documentclass[a4paper,oneside,10pt]{amsart}

\usepackage[ngerman,english]{babel}
\usepackage{a4wide}
\usepackage[T1]{fontenc}
\usepackage[ansinew]{inputenc}
\usepackage{lmodern} 
\usepackage{graphicx}
\usepackage{amsmath}
\usepackage{amsthm}
\usepackage{amsfonts}
\usepackage{amssymb}
\usepackage{setspace}
\usepackage{mathrsfs}
\usepackage[all]{xy}
\usepackage{enumerate}
\usepackage{soul}
\usepackage{pgf,tikz}
\usetikzlibrary{arrows, automata}
\usepackage{capt-of}
\usepackage{orcidlink}

\newtheorem{theorem}{Theorem}[section]

\newtheorem{corollary}[theorem]{Corollary}

\newtheorem{lemma}[theorem]{Lemma}
\newtheorem{proposition}[theorem]{Proposition}

\theoremstyle{definition}
\newtheorem{definition}[theorem]{Definition}

\newtheorem{remark}[theorem]{Remark}

\newtheorem*{theorem*}{Theorem}
\newtheorem*{proposition*}{Proposition}
\newtheorem*{definition*}{Definition}
\numberwithin{equation}{section}

\newenvironment{abc}{\begin{enumerate}[{\rm (a)}]}{\end{enumerate}}

\newenvironment{iiv}{\begin{enumerate}[{\rm (i)}]}{\end{enumerate}}

\def\V{\mathrm{V}}
\def\vv{\mathrm{v}}
\def\ee{\mathrm{e}}
\def\dd{\mathrm{d}}
\def\E{\mathrm{E}}
\def\W{\mathrm{W}}
\def\dom{\mathrm{D}}
\def\NN{\mathbb{N}}
\def\CC{\mathbb{C}}
\def\BB{\mathbb{B}}
\def\RR{\mathbb{R}}
\def\Ell{\mathrm{L}}
\def\diag{\mathrm{diag}}
\def\Id{\mathrm{I}}

\begin{document}
\title{Well-posedness of non-autonomous transport equation on metric graphs}

\author{Christian Budde}
\address{University of the Free State, Department of Mathematics and Applied Mathematics, P. O. Box 339, 9300 Bloemfontein, South Africa}
\email{buddecj@ufs.ac.za}

\author{Marjeta Kramar Fijav\v{z}} 
\address{University of Ljubljana, Faculty of Civil and Geodetic Engineering, Jamova 2, SI-1000 Ljubljana, Slovenia / Institute of Mathematics, Physics, and Mechanics, Jadranska 19, SI-1000 Ljubljana, Slovenia}
\email{marjeta.kramar@fgg.uni-lj.si}


\begin{abstract}
We consider transport processes on metric graphs with time-dependent velocities and show that, under  continuity assumption of the velocity coefficients,  the corresponding non-autonomous abstract Cauchy problem is well-posed  by means of evolution families and evolution semigroups.  
\end{abstract}

\date{\today}
\maketitle

\section{Introduction}

Consider a finite network (i.e., of pipelines)  where  some material is transported along its branches (i.e., pipes). The velocity of the transport  depends on a given branch but may also change in time. We would like to know under which condition such a system can be modelled in a way that for any given initial distribution we are able to predict the state of the system in any time. We would also like to obtain stable solutions, that continuously depend on the initial state. In this case we will call our problem well-posed.

\medskip
Such transport problems on networks have  already been studied by several authors. The operator theoretical approach by means of abstract Cauchy problems on Banach spaces was initiated by the second author and E.~Sikolya \cite{KS2005}, for an overview and further references  we refer to the survey \cite{KP2020}. However, the majority of the publications concentrates on time-independent transport and hence autonomous abstract Cauchy problems. A first attempt to non-autonomous problems of this kind was performed by F.~Bayazit et al.~\cite{BDF2013}. They considered transport on networks with boundary conditions changing in time. The advantage of such an approach is that the corresponding operator does not change its action on the Banach space only its domain changes in time. Our aim is to consider also the non-autonomous operator, that is, we study transport problems on finite metric graphs with time-dependent velocities along the edges. We use evolution families and evolution semigroups as studied by G.~Nickel \cite{N1997} and show that the abstract Cauchy problem which can be associated to the transport equation on these graphs is well-posed. 

\medskip
Let us also mention that diffusion and other processes on networks have also been studied by semigroup techniques, see e.g.~the monograph by D.~Mugnolo \cite{Delio2014}. Time-dependent diffusion on networks was considered in \cite{ADK2014} where a non-autonomous form method was used that is only applicable in Hilbert spaces.

\medskip
This paper is structured as follows. The first section consists of a  reminder on some notions from  graph theory as well as  of the definition and some results on general  non-autonomous abstract Cauchy problems and associated evolution families.  In the second section   we present our non-autonomous transport problem on a metric graph and rewrite it in an operator theoretical context. We first consider  non-autonomous operators $A(t)$ with a common time-independent domain {$\dom( A(t)) \equiv \dom$}. Acquistapace and Terreni \cite{AT1984} studied operators of this kind but their results  are not applicable in our situation since they assume the generation of analytic semigroups which is not our case. We can however apply the results by Kato   \cite{K1953}. In section \ref{nonaut} we treat the general case in which the operators $A(t)$ do not necessarily share a common domain. We will obtain the aimed for evolution family  as a composition of the translation and the multiplication semigroup, whereby the multiplication semigroup will be constructed by  the known solution semigroups of the autonomous case using the results by T.~Graser \cite{G1997}  on unbounded operator multipliers. In order to determine the domain of the generator of the evolution semigroup we will follow the approach via the extrapolation spaces as presented by R.~Nagel, G.~Nickel and S.~Romanelli \cite{NNR1996}. Here we will make use of the fact that in our case the extrapolated operators share a common extrapolation space. 

\section{Preliminaries}
\subsection{Metric graphs}  
We shall use  the notation presented for example in \cite{KS2005,BKR2017}. We take a simple,  connected, directed, finite graph $(\V,\E)$ with the set of \emph{vertices} $\V=\{\vv_1,\dots,\vv_n\}$ and the set of directed \emph{edges} $ \E = \{\ee_1,\dots,\ee_m\}\subseteq \V\times\V$.
We parametrize each edge with an interval $[0,1]$ and thus obtain a metric object called \emph{metric graph} (or network). By an abuse of notation we denote the endpoints of the edge $\ee=(\vv_i ,\vv_j)$ as $\vv_i = \ee(1)$ and $\vv_j = \ee(0)$. For technical reasons we assume that the graph is oriented contrary to the parametrisation of the edges. The structure of the graph can be described by its incidence matrices as follows. 
The  \emph{outgoing incidence matrix} $\Phi^{-}:=(\phi^{-}_{ij})_{n\times m}$  and \emph{incoming incidence matrix} $\Phi^{+}:=(\phi^{+}_{ij})_{n\times m}$  are defined as
\begin{equation*}\label{eq:incident}
 \phi^{-}_{ij} := 
 \begin{cases}
  1, & \text{if } \ee_j(1) = \vv_i,\\
  0, & \text{otherwise},
 \end{cases}
 \qquad 
\text{and}
\qquad 
 \phi^{+}_{ij} := 
 \begin{cases}
  1, & \text{if } \ee_j(0) = \vv_i,\\
  0, & \text{otherwise}.
 \end{cases}
\end{equation*} 

We will further assign some weights $0 \le  w_{ij} \le 1$ to the edges of our metric graph such that
\begin{equation}\label{eqn:noAbsorb}
\sum_{j=1}^m{w_{ij}}=1 \text{ for all } i=1,\dots, n.
\end{equation}
The so-called \emph{weighted (transposed) adjacency matrix of the line graph} $\mathbb{B} = (\mathbb{B}_{ij})_{m\times m}$ defined by 
\begin{equation}\label{eqn:adjMat}
\mathbb{B}_{ij}:=\begin{cases}
w_{ki},& \text{ if } \ee_j (0)=\vv_k =\ee_i(1),\\
0, & \text{ otherwise},
\end{cases}
\end{equation}
contains all information on our weighted metric graph. By \eqref{eqn:noAbsorb}, the matrix $\mathbb{B}$ is column stochastic, i.e., the sum of entries of each column is $1$, and defines a bounded positive operator on $\CC^m$ with $r(\mathbb{B})= \|\mathbb{B}\| =1$. Since the graph is connected, $\BB$ is an irreducible matrix. For additional terminology and properties we refer to \cite[Ch.~18]{BKR2017}.

\subsection{Non-autonomous abstract Cauchy problems}
The abstract operator-theoretical setting we will use is the setting of  so-called  non-autonomous abstract Cauchy problems. Let $(A(t),\dom(A(t)))_{t\in\RR}$ be a family of (unbounded) operators on a Banach space $X$. The associated \emph{non-autonomous abstract Cauchy problem} is given by
\begin{align}\tag{nACP}\label{eqn:nACP}
\begin{cases}
\dot{u}(t)=A(t)u(t),&\quad t\geq s,\\
u(s)=x \in X.
\end{cases}
\end{align}
The \emph{autonomous} case, i.e., when the operators $A(t)\equiv A$ do not depend on time, is well understood by means of strongly continuous one-parameter operator semigroups, or $C_0$-semigroups for short. There are several monographs on this theory, we refer to K.-J.~Engel and R.~Nagel \cite{EN}. 
To some extent this theory can be carried over to non-autonomous case, as described below. First, let us recall the notions of (classical) solutions and well-posedness of \eqref{eqn:nACP}.

\begin{definition}{\cite[Def. VI.9.1]{EN}}
Let $(A(t),\dom(A(t)))$, $t\in\RR$, be linear operators on a Banach space $X$ and take $s\in\RR$ and $x\in\dom(A(s))$. A \emph{(classical) solution of \eqref{eqn:nACP}} is a function $u(\cdot,s,x)=u\in\mathrm{C}^1\left(\left[s,\infty\right),X\right)$ such that $u(t)\in\dom(A(t))$ and $u$ satisfies \eqref{eqn:nACP} for $t\geq s$. The Cauchy problem \eqref{eqn:nACP} is called \emph{well-posed} (on spaces $Y_t$) if there are subspaces $Y_s\subseteq\dom(A(s))$, $s\in\RR$, dense in  $X$, such that for $s\in\RR$ and $x\in Y_s$ there is a unique solution $t\mapsto u(t,s,x)\in Y_t$ of \eqref{eqn:nACP}. In addition, for $s_n\to s$ and $Y_{s_n}\ni x_n\to x$, we have $\widetilde{u}(t,s_n,x_n)\to\widetilde{u}(t,s,x)$ uniformly for $t$ in compact intervals in $\RR$, where we set $\widetilde{u}(t,s,x):={u}(t,s,x)$ for $t\geq s$ and $\widetilde{u}(t,s,x):=x$ for $t<s$.
\end{definition}

In the autonomous case, the solutions are represented by $C_0$-semigroups. An appropriate analoge for the solutions of \eqref{eqn:nACP} are  so-called evolution families.

\begin{definition}{\cite[Def. VI.9.2]{EN}}
A family of bounded operators $(U(t,s))_{t,s\in\RR, t\geq s}$ on a Banach space $X$ is called a \emph{(strongly continuous) evolution family} if
\begin{iiv}
	\item $U(t,s)=U(t,r)U(r,s)$ and $U(s,s)=\Id$ for $t\geq r\geq$ s and $t,r,s\in\RR$,
	\item the mapping $\left\{(\tau,\sigma)\in\RR^2:\ \tau\geq\sigma\right\}\ni(t,s)\mapsto U(t,s)$ is strongly continuous,
	\item there exists $M\geq1$ and $\omega\in\RR$ such that $\left\|U(t,s)\right\|\leq M\ee^{\omega(t-s)}$ for all $t\geq s$.
\end{iiv}
We say that $(U(t,s))_{t\geq s}$ solves the Cauchy problem \eqref{eqn:nACP} (on space $Y_t$) if there are dense subspaces $Y_s\subseteq X$, $s\in\RR$,  such that $U(t,s)Y_s\subseteq Y_t\subseteq\dom(A(t))$ for $t\geq s$ and the function $t\mapsto U(t,s)x$ is a solution of \eqref{eqn:nACP} for $s\in\RR$ and $x\in Y_s$.
\end{definition} 

\begin{proposition}{\cite[Prop.~2.5]{N1997}}
The Cauchy problem \eqref{eqn:nACP} is well-posed (on $Y_t$) if and only if there is an evolution family solving \eqref{eqn:nACP} (on $Y_t$).
\end{proposition}

A uniform generation theorem in the style of Hille--Yosida is unfortunately not known for non-autonomous Cauchy problems. In fact, the existence of solutions of \eqref{eqn:nACP} is a priori not clear. However, there are several attempts to the characterisation the  well-posedness of certain classes of non-autonomous abstract Cauchy problems, let us only mention the work by P.~Acquistapace and B.~Terreni \cite{AT1987}, T.~Kato and H.~Tanabe \cite{KT1962}, as well as the survey by R.~Schnaubelt \cite{S2002}. There is one possible approach to the  well-posedness results that applies operator semigroup theory  by means of the so-called evolution semigroups, as follows. To any evolution family $(U(t,s))_{t\geq s}$ on a Banach space $X$ one can associate an operator semigroup $(\mathcal{T}(t))_{t\geq0}$ on the space $\mathrm{C}_0(\RR,X)$ by
\[
(\mathcal{T}(t)f)(s):= U(s,s-t)f(s-t),\quad t\geq0,\ f\in\mathrm{C}_0(\RR,X),\ s\in\RR.
\]
By \cite[Lemma~VI.9.10]{EN} this semigroup is strongly continuous on $\mathrm{C}_0(\RR,X)$ and is called the \emph{evolution semigroup}. We denote its generator by $(G,\dom(G))$. The following result characterizes evolution semigroups. By $(T_\mathrm{r}(t))_{t\ge 0}$ we denote the right translation semigroup on $\mathrm{C}_{\mathrm{c}}(\RR)$ defined by $(T_\mathrm{r}(t)\varphi )(s):= \varphi(s-t)$.

\begin{theorem}\label{thm:CharEvoSemi}{ \cite[Thm.~VI.9.14]{EN}}
Let $(T(t))_{t\geq0}$ be a $C_0$-semigroup on $\mathrm{C}_0(\RR,X)$ with generator $(G,\dom(G))$. Then following assertions are equivalent.
\begin{abc}
	\item $(T(t))_{t\geq0}$ is an evolution semigroup.
	\item $T(t)(\varphi f)=(T_\mathrm{r}(t)\varphi)T(t)f$ for all $\varphi\in\mathrm{C}_{\mathrm{c}}(\RR)$, $f\in\mathrm{C}_0(\RR,X)$ and $t\geq0$.
	\item For $f\in\dom(G)$ and $\varphi\in\mathrm{C}_{\mathrm{c}}^1(\RR)$ we have $\varphi f\in\dom(G)$ and $G(\varphi f)=\varphi Gf-\varphi'f$.
\end{abc}
\end{theorem}

Finally, one can characterize the well-posedness of \eqref{eqn:nACP} as follows.

\begin{theorem}\label{thm:WPEvoSemi}{\cite[Thm.~2.9]{N1997}}
Let $(A(t),\dom(A(t)))_{t\in\RR}$ be a family of linear operators on a Banach space $X$. The following are equivalent.
\begin{abc}
	\item \eqref{eqn:nACP} is well-posed for the family of operators $(A(t),\dom(A(t)))_{t\in\RR}$.
	\item There exists a unique evolution semigroup $(T(t))_{t\geq0}$ with generator $(G,\dom(G))$ and a $T(t)$-invariant core $D\subseteq\mathrm{C}_0^1(\RR,X)\cap\dom(G)$ such that $Gf+f'=A(\cdot)f$ for all $f\in D$.
\end{abc}
\end{theorem}

The above result shows that determining the domain of the generator of the evolution semigroup is the most important step on the way to solve a given non-autonomous abstract Cauchy problem.

\section{The non-autonomous network transport problem}
\subsection{The setting}
Let us now consider the following non-autonomous dynamical system taking place along $m$ edges of a metric graph. 
\begin{align}\label{eqn:nF}
\tag{nF}
\begin{cases}
\frac{\partial}{\partial t}u_j(x,t)=c_j(t)\frac{\partial}{\partial x}u_j(x,t),&\quad x\in\left(0,1\right),\ t\geq s,\\
u_j(x,s)= f_j(x),&\quad x\in\left(0,1\right),\\ 
\phi^-_{ij}c_j(t)u_j(1,t)= w_{ij}\sum_{k=1}^m{\phi^+_{ik}c_k(t)u_k(0,t)},&\quad t\geq s,
\end{cases}
\end{align}
for $i=1,\dots,n$, $j=1,\dots, m$.
The first equation models the transport along the edge $\ee_j$ where $c_j(t)$ is the time-variable velocity coefficient along this edge. {In what follows, we will assume that $c_j(t)>0$ and that there exist {$m>0$ and $M>0$ }such that $m\leq c_j(t)\leq M$ for all $j=1,\ldots,m$ and all $t\in\RR$.} The second equation gives the initial mass distribution along the edges while the third equation represents the boundary conditions  in the vertices of the graph. Here, the graph structure is encoded in terms of incidence matrices. Observe, that the weights $w_{ij}$ give proportions of the material arriving in the vertex $\vv_i$ that is distributed to the edge $\ee_j$. By \eqref{eqn:noAbsorb}, at all times the mass is conserved, i.e., the Kirchhoff law holds in all vertices.

\medskip
We will show that there exists a solution of the non-autonomous set of equations \eqref{eqn:nF} by means of evolution families. The equivalent problem in the autonomous case, i.e., when all $c_j(t)\equiv c_j$ are independent of time has already been considered by several authors, cf.~\cite{KS2005,KP2020}. A first approach to time-dependent transport equations on networks was done by F.~Bayazit, B.~Dorn and M.~Kramar Fijav\v{z} \cite{BDF2013}.
In particular, they studied transport processes with time-dependent weights $\omega_{ij}(t)$ but constant velocities $c_j(t)\equiv 1$. This can be interpreted as autonomous transport with time-depending boundary conditions (i.e., the structure of the graph changes in time). They were able to obtain well-posedness as well as some results on the asymptotic behaviour. In our case, already the transport equation on each edge is non-autonomous and a different approach is needed.

\medskip
We use operator theoretical approach and first define  Banach space 
\[
X:=\Ell^1\left(\left[0,1\right],\CC^m\right)\cong\Ell^1\left(\left[0,1\right],\CC\right)^m,
\]
with the usual norm
\[
\left\|f\right\|_X:=\sum_{j=1}^m\int_0^1{\left|f_j(t)\right|\ \dd{t}}.
\]
On this Banach space $X$ we consider the  family of operators $(A(t),\dom(A(t)))$ defined by
\begin{align}\label{eqn:DiffOp}
A(t)&:=\begin{pmatrix}
c_1(t)\frac{\dd}{\dd{x}}& &0\\
& \ddots &\\
0& & c_m(t)\frac{\dd}{\dd{x}}
\end{pmatrix},\\
\dom(A(t))&:=\left\{u\in \W^{1,1}\left(\left[0,1\right],\mathbb{C}^m\right):\ u(1)=\mathbb{B}_C(t)u(0)\right\}
\label{eqn:DiffOp-D}
\end{align}
where the matrix $\BB_C(t)$ is defined by $\BB_C(t):=C(t)^{-1}\BB C(t)$ for $C(t):=\diag(c_j(t))_j$ and $\BB$ is the adjacency matrix given in \eqref{eqn:adjMat}. {We emphasize that $C(t)^{-1}$ exists for all $t\in\RR$ as we assumed that there exist {$m>0$ and $M>0$} such that $m\leq c_j(t)\leq M$ for all $j=1,\ldots,m$ and all $t\in\RR$.} We can now rewrite  our transport problem \eqref{eqn:nF} in the form of an abstract Cauchy problem as follows.    

\begin{lemma}
The abstract non-autonomous Cauchy problem  
\begin{align}\label{nACP-F}
\begin{cases}
\dot{u}(t)=A(t)u(t),&\quad t,s\in\RR,\ t\geq s,\\
u(s)=(f_j)_{j=1}^m,
\end{cases}
\end{align}
associated to the family of operators $(A(t),\dom(A(t)))$ defined in \eqref{eqn:DiffOp}-\eqref{eqn:DiffOp-D} is equivalent to the transport problem stated in \eqref{eqn:nF}.
\end{lemma}

\begin{proof}
We only need to check that the condition in the domains $\dom(A(t))$ is equivalent to the boundary condition of \eqref{eqn:nF}. 
This can be done the same way as  in the proof of \cite[Prop.~18.2]{BKR2017}.
\end{proof}

Our main goal is to show, that problem \eqref{nACP-F} is well-posed. We start with the simplest  situation.

\subsection{Operators with constant domains}\label{sec:ConstDom}

 In the literature treating non-autonomous Cauchy problems it is often assumed that the domains of the operators $A(t)$  appearing in \eqref{nACP-F}

do not depend on $t$.  In this case, the smoothness of the coefficients $c_j(\cdot)$ yields the following well-posedness result.
 
 \begin{proposition}
Let $(A(t),\dom(A(t)))$ be the operator family defined in \eqref{eqn:DiffOp}-\eqref{eqn:DiffOp-D}. Assume that $\dom(A(t)) = \dom(A(0)) =:D$ for all $t\in\RR$ and that $C(\cdot)\in\mathrm{C}^1(\RR,\RR^m)$. Then  \eqref{eqn:nF}
is a well-posed problem.
\end{proposition}
 
\begin{proof}
This follows from the well-known result by Kato once we verify the assumptions of \cite[Thm.~5]{K1953}. 
First, note that by \cite[Cor.~18.15]{BKR2017}, for each fixed $t\in\RR$, operator $(A(t),D)$ generates a contractive $C_0$-semigroup.
Next, the smoothness of coefficients $c_j(\cdot)$ implies, that the mapping $t\mapsto A(t)f$ is continuously differentiable for each $f\in D$.  Finally, the latter condition is  known  to be equivalent to the Kato's assumptions,  see \cite[Sect.~VI.9.5]{EN} or \cite[Prop.~2.1]{SG2014}.
\end{proof} 
 
Let us state some simple conditions that yield $t$-independence of the domains $\dom(A(t))$.

\begin{lemma}\label{lemma:ConstDom}
The following two properties of the coefficients $c_j(\cdot)$ appearing in  \eqref{eqn:DiffOp}-\eqref{eqn:DiffOp-D} are equivalent.
\begin{enumerate}[(i)]
\item Whenever $\mathbb{B}_{ij}\neq 0 $,
$c_i(t_1) c_j(t_1)^{-1} = c_i(t_2) c_j(t_2)^{-1}$ for any  $t_1,t_2 \in\RR$.
\item  $C(t)=\alpha(t) D$, $t\in\RR$, for some scalar function $\alpha(\cdot)$ and diagonal matrix $D$.
\end{enumerate}
Moreover, each of these properties implies  that $\dom(A(t_1)) = \dom(A(t_2))$ for any  $t_1,t_2 \in\RR$.
\end{lemma}

\begin{proof}
It is easy to see that (i) implies (ii). Indeed, by taking $t_1=t$ and $t_2=0$, and denoting $D=\diag(d_i)_i:=C(0)$, we can reformulate (i) as $c_i(t) d_i^{-1}=c_j(t) d_j^{-1}=:\alpha (t)$, which yields  (ii).
We further observe that $\left(\BB_C(t)\right)_{ij} = c_i^{-1}(t) \BB_{ij} c_j(t)$. Hence, (ii)  clearly implies $\BB_C(t_1) = \BB_C(t_2)$ and thus also (i). 
\end{proof}

One can interpret condition (i) in the above Lemma as follows: whenever there is an inflow into the $i$-th edge from the $j$-th edge, at all times the velocities of the flow should stay in the same ratio which is, for example,  determined by the radii of the respective pipelines. This is reasonable in many situations. However, our main result will show that neither constant domains nor the  $C^1$-condition on the coefficients $c_j(\cdot)$ is necessary for the well-posedness of  \eqref{eqn:nF}.

\section{Well-posedness of  the general non-autonomous problem}\label{nonaut}

{Let us now consider the general situation when the domains of operators $A(t)$  are not necessarily constant. We start by observing that the domains of the adjoint operators $A'(t)$, however, do not depend on $t$.
\begin{lemma}\label{lem:adj}
The adjoints of the operators  $(A(t),\dom(A(t)))$ defined in \eqref{eqn:DiffOp}-\eqref{eqn:DiffOp-D} are given  by
\begin{align*}\
A'(t)&:=\begin{pmatrix}
- c_1(t)\frac{\dd}{\dd{x}}& &0\\
& \ddots &\\
0& &- c_m(t)\frac{\dd}{\dd{x}}
\end{pmatrix},\\
\dom(A'(t))&:=\left\{v\in \W^{1,\infty}\left(\left[0,1\right],\mathbb{C}^m\right):\ v(0)=\mathbb{B}^{\top}v(1)\right\}.
\label{eqn:DiffOp-DAdj}
\end{align*}
\end{lemma}
\begin{proof}
For $u\in \dom(A(t))$ and $v\in \W^{1,\infty}\left(\left[0,1\right],\mathbb{C}^m\right)$ we first compute
 \begin{align*}
 \langle A(t)  u,v\rangle &= \sum_{k=1}^m \int_0^1 c_k(t) u'_k(x) v_k(x) \; dx \\
 &= \sum_{k=1}^m  c_k(t) \left(u_k(1) v_k(1) -  u_k(0) v_k(0) \right) -\sum_{k=1}^m \int_0^1 c_k(t) u_k(x) v'_k(x) \; dx.
 \end{align*}
 Now we observe that $v \in \dom(A'(t))$ if and only if 
 \begin{align*}
 0 &=   \sum_{k=1}^m  c_k(t) \left(u_k(1) v_k(1) -  u_k(0) v_k(0) \right) \\
 &=\sum_{k=1}^m  c_k(t)  \left( \sum_{i=1}^m \left(\left(\BB_C(t)\right)_{ki} u_i(0) \right)v_k(1) -  u_k(0) v_k(0) \right)  \\
  &=  \sum_{i=1}^m \left( \sum_{k=1}^m  \left(\BB_C(t)^{\top}\right)_{ik} c_k(t)  v_k(1) \right) u_i(0)  -  \sum_{k=1}^m c_k(t) v_k(0) u_k(0)
  \end{align*}
  for all  $u\in \dom(A(t))$. By using $\BB_C(t) ^{\top} C(t) = C(t) \BB^{\top} $ we  obtain that this is further equivalent to 
\[C(t) \BB^{\top}  v(1) - C(t) v(0)=0\]
and since $C(t)$ is invertible, we are done.
\end{proof}
}

We {shall now employ} evolution families and evolution semigroups approach to show the well-posedness of \eqref{eqn:nACP} 
for operators $A(t)$ given by \eqref{eqn:DiffOp} with non-autonomous domains of the form \eqref{eqn:DiffOp-D}.{We proceed in several steps.}

\subsection{The associated multiplication semigroup.}

By \cite[Cor.~18.15]{BKR2017}, for each fixed $s\in\RR$, the abstract Cauchy problem corresponding to the single operator $(A(s),\dom(A(s)))$ is well-posed {on $X$} and the solution is given by a positive contraction semigroup $(T_s(t))_{t\geq0}$. From this we construct a new operator semigroup $(\mathcal{S}(t))_{t\geq0}$ on the vector-valued function space $\mathrm{C}_0(\RR,X)$ by 
\begin{equation}\label{eq:mult-sgr}
(\mathcal{S}(t)f)(s):=T_s(t)f(s),\quad t\geq0,\ f\in\mathrm{C}_0(\RR,X),\ s\in\RR.
\end{equation}
We will show that $(\mathcal{S}(t))_{t\geq0}$ is  a strongly continuous semigroup on  $\mathrm{C}_0(\RR,X)$ and give its generator.

\begin{proposition}\label{prop:MultiSemi}
Assume that $c_j(\cdot)\in\mathrm{C}(\RR)$ and that there exist {$m>0$ and $M>0$} such that $m\leq c_j(t)\leq M$ for each $t\in\RR$. Then, the operator semigroup $(\mathcal{S}(t))_{t\geq0}$ is strongly continuous on $\mathrm{C}_0(\RR,X)$ and its generator $(\mathcal{A},\dom(\mathcal{A}))$ is given by the multiplication operator on $\mathrm{C}_0(\RR,X)$ induced by the family of differential operators on {$X=\Ell^1\left(\left[0,1\right],\mathbb{C}^m\right)$} defined in \eqref{eqn:DiffOp}-\eqref{eqn:DiffOp-D}, i.e.,
\begin{align*}
(\mathcal{A}f)(s)&=A(s)f(s),\\
\dom(\mathcal{A})&=\left\{f\in\mathrm{C}_0(\RR,X): \  f(s)\in\dom(A(s)) \ { \forall s\in\RR \ \text{and} }  \ A(\cdot)f(\cdot)\in\mathrm{C}_0(\RR,X)\right\}.
\end{align*}
In addition, the set $\rho(\mathcal{A})\cap\rho(A(s))$ is nonempty for each $s\in\RR$ and the resolvent $R(\lambda,\mathcal{A})$ is a bounded operator multiplier for all $\lambda\in\rho(\mathcal{A})$, given by 
\[\left(R(\lambda,\mathcal{A})f\right)(s) = R(\lambda, A(s)) f(s),\quad s\in\RR, f\in\mathrm{C}_0(\RR,X).\]
\end{proposition}

\begin{proof}
The assertion follows directly by \cite[Thm.~3.4 \& Lem.~3.5]{G1997} once we show that the map $\RR\times\RR_{+}\ni(s,t)\mapsto T_s(t)$ is strongly continuous. 
To this end we first prove that the mapping $s\mapsto R(\lambda,A(s))$ is strongly continuous. For any fixed $s\in\RR$  we can use  \cite[Prop.~18.12]{BKR2017}  and obtain an explicit formula for the resolvent $R(\lambda,A(s))$. In particular,  for $\mathrm{Re} \lambda >0$  we have
\[
R(\lambda,A(s))=\left(\Id+E_\lambda(\cdot,s)(1-\mathbb{B}_{C,\lambda}(s))^{-1}\mathbb{B}_{C,\lambda}(s)\otimes\delta_0\right)R_{\lambda}(s),
\]
where
\[
E_{\lambda}(\tau,s)=\diag\left(\ee^{(\lambda/c_k(s))\tau}\right),\quad 
\mathbb{B}_{C,\lambda}(s)=E_\lambda(-1,s)\mathbb{B}_C(s),
\]
$\delta_0$ is the point evaluation at $0$, and
\[
(R_\lambda(s)f)(\tau)=\int_{\tau}^1{E_\lambda(\tau-\xi,s)C^{-1}(s)f(\xi)\ \dd\xi}.
\]
It is easy to see, by use of the explicit resolvent formula, that by our continuity assumptions on $c_j(\cdot)$, the mapping $s\mapsto R(\lambda,A(s))f$ is continuous for every $f\in\mathrm{C}_0(\RR,X)$. Now we estimate 
 \begin{align*}
\|T_{s_0}(t_0)f(s_0) - T_s(t)f(s)\| &\le  \|T_{s_0}(t_0)f(s_0) - T_{s_0}(t)f(s_0)\| \\
  &+ \|T_{s_0}(t)f(s_0) - T_s(t)f(s_0)\|+ \|T_{s}(t)f(s_0) - T_s(t)f(s)\|
 \end{align*}
 and note that the first term tends to $0$ when $t\to t_0$ since the semigroup $(T_{s_0}(t))_{t\geq0}$ is strongly continuous while the third  term tends to $0$ when $s\to s_0$ since $(T_{s}(t))_{t\geq0}$ is contractive and $f$ is continuous. Finally, due to the strong continuity of $s\mapsto R(\lambda,A(s))$ we can apply Trotter--Kato theorem  \cite[Thm.~III.4.9]{EN} and see that also the second term converges to $0$ as $s\to s_0$.
\end{proof}

\subsection{The associated evolution semigroup and its generator}

Inspired by the work of R.~Nagel, G.~Nickel and S.~Romanelli \cite[Sect.~4]{NNR1996} we define another semigroup\footnote{This is unfortunately false in general, see Erratum on p.~\pageref{erratum}.} $(\mathcal{T}(t))_{t\geq0}$ on $\mathrm{C}_0(\RR,X)$ by
\begin{equation}\label{eq:sgr-evol}
(\mathcal{T}(t)f)(s):={(\mathcal{S}(t)f)}(s-t),\quad t\geq0,\ f\in\mathrm{C}_0(\RR,X),\ s\in\RR,
\end{equation}
where $(\mathcal{S}(t))_{t\geq0}$ is the multiplication semigroup given in \eqref{eq:mult-sgr}.
Hence, $(\mathcal{T}(t))_{t\geq0}$ is the composition of the translation semigroup and the multiplication semigroup.  An application of Theorem \ref{thm:CharEvoSemi} shows that $(\mathcal{T}(t))_{t\geq0}$ is an evolution semigroup. 

\medskip
Our next goal is to determine the generator of this semigroup. This can be done, as already mentioned in \cite[Sect.~4]{NNR1996}, by means of extrapolation spaces. Recall, that the first extrapolation space of  Banach space $X$ with respect to  semigroup generator $(A,\dom(A))$ is defined to be the completion of $X$ with respect to a new norm $\left\|\cdot\right\|_{-1}$ on $X$ defined by
\begin{equation}\label{eq:norm-1}
\left\|x\right\|_{-1}:=\left\| {R({\lambda_0}, A)} x\right\|,\quad x\in X
\end{equation}
{for some ${\lambda_0}\in\rho(A)$ (usually one may assume that ${\lambda_0} =0$)}. This completion is denoted by $X_{-1}$ and is called the \emph{first extrapolation space}. If we want to stress the dependence on the operator $(A,\dom(A))$ we write $X_{-1}^{(A)}$. By continuity, we can extend $(T(t))_{t\geq0}$ to a semigroup of linear operators on {$X_{-1}^{(A)}$}. This semigroup is denoted by $(T_{-1}(t))_{t\geq0}$. The first extrapolation space and the corresponding extrapolated semigroup have the following properties.

\begin{proposition}{\cite[Thm.~II.5.5]{EN}}\label{prop:extra}
Let $(T(t))_{t\geq0}$ be a strongly continuous semigroup on $X$ with generator $(A,\dom(A))$. The following assertions hold true.
\begin{iiv}
	\item $X$ is dense in $X_{-1}^{(A)}$.
	\item $(T_{-1}(t))_{t\geq0}$ is a strongly continuous semigroup on $X_{-1}^{(A)}$.
	\item The generator $A_{-1}$ of $(T_{-1}(t))_{t\geq0}$ has domain $\dom(A_{-1})=X$. 
	\item {For $\lambda_0\in\rho(A)$, which is used to define the norm $\left\|\cdot\right\|_{-1}$ in \eqref{eq:norm-1},  $\lambda_0-A_{-1}\colon X\rightarrow X_{-1}^{(A)}$ is the unique extension of $\lambda_0-A\colon\dom(A)\rightarrow X$ to an isometry.}
\end{iiv}
\end{proposition}

Typical examples for such extrapolation spaces are Sobolev spaces and weighted $\mathrm{L}^p$-spaces. For more explicit examples we refer to \cite[Ex.~II.5.7,5.8]{EN}, \cite[Sect.~3]{NNR1996} and \cite[Sect.~5]{BF2019}. Similar results have been recently discovered for Bochner $\mathrm{L}^p$-spaces as well, cf.~\cite{BH2022}. Recall that by \cite[Thm.~4.7]{G1997}, the first extrapolation space of $\mathrm{C}_0(\RR,X)$ with respect to the multiplication operator $(\mathcal{A},\dom(\mathcal{A}))$ looks like
\[
\left[\mathrm{C}_0(\RR,X)\right]_{-1}^{\mathcal{A}}=\left\{f\in\prod_{s\in\RR}{{X_{-1}^{A(s)}}}:\ f\ \text{is fiber-continuous and vanishes at infinity}\right\}.
\]
Here, for fixed  $s\in\RR$, $X_{-1}^{A(s)}$ denotes the first extrapolation space of $X$ with respect to the operator $(A(s),\dom(A(s)))$. A function $f\in\prod_{s\in\RR}{X_{-1}^{A(s)}}$ is called \emph{fiber-continuous} if for any $s_0\in\RR$ and $\varepsilon>0$ there exist $x\in X$ and $\delta>0$ such that $\left|s_0-s\right|<\delta$ implies that 
$$\left\|f(s_0)-x\right\|_{-1,s_0}+\left\|f(s)-x\right\|_{-1,s}<\varepsilon,$$
see \cite[Def.~4.4]{G1997}. Moreover, such a function \emph{vanishes at infinity} if $\lim\limits_{\left|s\right|\to\infty}{\left\|f(s)\right\|_{-1,s}}=0$. 
 
\medskip 
At a first glance, it seems that the elements of the family of extrapolation spaces $(X_{-1}^{A(s)})_{s\in\RR}$ may differ from each other and hence the product of the extrapolated translation semigroup and the extrapolated multiplication semigroup may not be well-defined. However, in our case the operators $(A(s),\dom(A(s)))$, $s\in\RR$, have a special appearance.
  {Let ${\lambda_0}\in\CC$ with $\mathrm{Re}({\lambda_0})>0$. By \cite[Cor.~18.13]{BKR2017} one has {${\lambda_0}\in\rho(A(s))$ for all $s\in\RR$.} 
{Now, by \cite[Rem.~6.6(c)]{A1997}, Lemma \ref{lem:adj},} and an application of \cite[Exercise~{II}.5.9(1)]{EN}, 
it follows, that all extrapolation spaces $X_{-1}^{A(s)}$ are equivalent, i.e., 
\begin{equation*}\label{X-1}
X_{-1}^{A(s)}\cong X_{-1}^{A(0)}=:X_{-1} \text{  for all } s\in\RR
\end{equation*}
and there exists $\eta>0$ such that $\frac{1}{\eta}\left\|x\right\|_{X_{-1}}\leq\left\|x\right\|_{X_{-1}^{A(s)}}\leq\eta\left\|x\right\|_{X_{-1}}$ for all $s\in\RR$.} Hence, we have that {$\left[\mathrm{C}_0(\RR,X)\right]_{-1}^{\mathcal{A}}=\mathrm{C}_0(\RR,X_{-1})$.}
By following the procedure in \cite[Sect.~4]{NNR1996}, we 
 can then describe the domain $\dom(\mathcal{G})$ of the generator of the evolution semigroup $(\mathcal{T}(t))_{t\geq0}$.

\begin{theorem}\label{prop:DomEvoSemi}\footnote{This is unfortunately false in general, see Erratum on p.~\pageref{erratum}.}
Assume that $c_j(\cdot)\in\mathrm{C}(\RR)$ and that there exist {$m>0$ and $M>0$} such that $m\leq c_j(t)\leq M$ for each $t\in\RR$.
The generator $(\mathcal{G},\dom(\mathcal{G}))$ of the $C_0$-semigroup $(\mathcal{T}(t))_{t\geq0}$ on $\mathrm{C}_0(\RR,X)$ defined in \eqref{eq:sgr-evol} is given by
\begin{equation}\label{eqn:DomEvoSemi}
\begin{aligned}
(\mathcal{G}f)(s)&=A_{-1}(s)f(s)-f'(s),\\
\dom(\mathcal{G})&=\left\{f\in\mathrm{C}_0(\RR,X)\cap\mathrm{C}_0^1(\RR,X_{-1}):\ {\mathcal{A}_{-1}} f-f'\in\mathrm{C}_0(\RR,X)\right\}.
\end{aligned}
\end{equation}
\end{theorem} 

{
\begin{proof}
As mentioned above, one has that $\left[\mathrm{C}_0(\RR,X)\right]_{-1}^{\mathcal{A}}=\mathrm{C}_0(\RR,X_{-1})$. On this space we consider the extrapolated multiplication semigroup $(\mathcal{S}_{-1}(t))_{t\geq0}$ which is defined by
\[
(\mathcal{S}_{-1}(t)f)(s):=T_{-1,s}(t)f(s),
\]
where $(T_{-1,s}(t))_{t\geq0}$ is the extrapolated semigroup for  $(T_s(t))_{t\geq0}$  on  $X_{-1}$.
Its generator is  denoted by $\mathcal{A}_{-1}$ and is by  \cite[Lem.~4.6]{G1997} given as
\[\left(\mathcal{A}_{-1}f\right)(s)= A_{-1}(s)f(s),\quad \dom(\mathcal{A}_{-1}) = \mathrm{C}_0(\RR,X).\]
On the space $\mathrm{C}_0(\RR,X_{-1})$ we also consider the operator $(\mathcal{B},\dom(\mathcal{B}))$ defined by
\[
\mathcal{B}f:=f',\quad \dom(\mathcal{B}):=\mathrm{C}^1_0(\RR,X_{-1}),
\]
which  generates the right translation semigroup on $\mathrm{C}_0(\RR,X_{-1})$. Now, define the {evolution} semigroup $(\widetilde{\mathcal{T}}(t))_{t\geq0}$ on $\mathrm{C}_0(\RR,X_{-1})$ by
\[
(\widetilde{\mathcal{T}}(t)f)(s):={\left(\mathcal{S}_{-1}(t)f\right)}(s-t),\quad t\ge 0, s\in \RR.
\]
{Since our multiplication and translation semigroups commute,}
one obtains for the generator $(\widetilde{\mathcal{G}},\dom(\widetilde{\mathcal{G}}))$ {of the multiplication semigroup  $(\widetilde{\mathcal{T}}(t))_{t\geq0}$} that for every $f\in\dom(\mathcal{A}_{-1})\cap\dom(\mathcal{B})${, which is a core for $\widetilde{\mathcal{G}}$,} one has 
\[
(\widetilde{\mathcal{G}}f)(s)=-f'(s)+(\mathcal{A}_{-1}f)(s)=-f'(s)+A_{-1}(s)f(s),\quad s\in\RR.
\]
{see \cite[Sec.~II.2.7]{EN}. }
By observing that $(\mathcal{T}(t))_{t\geq0}$ given in \eqref{eq:sgr-evol} coincides with  $(\widetilde{\mathcal{T}}(t))_{t\geq0}$ {restricted to the space} $\mathrm{C}_0(\RR,X)$, one has that {its generator $\mathcal{G}$ is the part of $\widetilde{\mathcal{G}}$ in $\mathrm{C}_0(\RR,X)$, that is,}
\[
\dom(\mathcal{G})=\{f\in\mathrm{C}_0(\RR,X)\cap\dom(\widetilde{\mathcal{G}}):\ \widetilde{\mathcal{G}}f\in\mathrm{C}_0(\RR,X)\}
\]
and $\mathcal{G}f=\widetilde{\mathcal{G}}f$ for all $f\in\dom(\mathcal{G})$, see \cite[Sec.~II.2.3]{EN}. {It only remains to specify the domain $\dom(\mathcal{G})$. We already know that 
\[\mathrm{C}_0(\RR,X_{-1}) \supset  \dom(\widetilde{\mathcal{G}}) \supset \dom(\mathcal{A}_{-1})\cap\dom(\mathcal{B}) =  \mathrm{C}_0(\RR,X) \cap \mathrm{C}^1_0(\RR,X_{-1}).\]
Now take $f\in\mathrm{C}_0(\RR,X)\cap\dom(\widetilde{\mathcal{G}}) =  \dom(\mathcal{A}_{-1})\cap\dom(\widetilde{\mathcal{G}})$.} Then both $\lim\limits_{t\to0}{\frac{\mathcal{S}_{-1}(t)f-f}{t}}$ and $\lim\limits_{t\to0}{\frac{\widetilde{\mathcal{T}}(t)f-f}{t}}$ exist in $\left[\mathrm{C}_0(\RR,X)\right]_{-1}^{\mathcal{A}}=\mathrm{C}_0(\RR,X_{-1})$. We notice that
\begin{align}
&\lim_{t\to0}{\frac{(\widetilde{\mathcal{T}}(t)f-f)(s)}{t}}-\lim_{t\to0}{\frac{(\mathcal{S}_{-1}(t)f-f)(s)}{t}}\notag \\
&=\lim_{t\to0}{\frac{({\mathcal{S}}_{-1}(t)f)(s-t)-f(s)}{t}}-\lim_{t\to0}{\frac{(\mathcal{S}_{-1}(t)f)(s)-f(s)}{t}}\notag \\
&=\lim_{t\to0}{\frac{T_{-1,s-t}(t)f(s-t)-f(s)}{t}}-\lim_{t\to0}{\frac{T_{-1,s}(t)f(s)-f(s)}{t}}\notag \\
&=\lim_{t\to0}{\frac{T_{-1,s-t}(t)f(s-t)-f(s-t)}{t}} -\lim_{t\to0}{\frac{T_{-1,s}(t)f(s)-f(s)}{t}} + \lim_{t\to0}{\frac{f(s-t)-f(s)}{t}}.\label{eq:last}
\end{align}
Observe, that $f(\tilde{s}) \in \dom(A_{-1}(\tilde{s}))=X$ for each $\tilde{s}\in\RR$ and the operator $A_{-1}(\tilde{s})$ generates $(T_{-1,\tilde{s}}(t))_{t\geq0}$ on $X_{-1}$. {Moreover,  we will now show that the mapping $\tilde{s}\mapsto A_{-1}(\tilde{s})x$ in continuous for all $x\in X$. Indeed, let $s,\tilde{s}\in\RR$ and $x\in X$. First recall that $n\in\rho(A(s))$ for all $s\in\RR$  (see \cite[Cor.~18.13]{BKR2017}) and we can define two sequences $(x_n)_{n\in\NN}$ and $(\tilde{x}_n)_{n\in\NN}$,
\[
x_n:=nR(n,A_{-1}(s))x=nR(n,A(s))x\quad\text{and}\quad\tilde{x}_n:=nR(n,A_{-1}(\tilde{s}))x=nR(n,A(\tilde{s}))x.
\]
By \cite[Lem.~II.3.4(i)]{EN} we know that $x_n\to x$ and $\tilde{x}_n\to x$ as $n\to\infty$. We also see that by the resolvent identity one gets
\begin{equation}\label{eq:res}
A_{-1}(s)x_n=n^2R(n,A(s))x-nx\quad\text{and}\quad A_{-1}(\tilde{s})\tilde{x}_n=n^2R(n,A(\tilde{s}))x-nx.
\end{equation}
Fix some $\lambda_0\in\rho(A(s))$, which is used to define the norm $\left\|\cdot\right\|_{-1}$ in \eqref{eq:norm-1}, and let $\varepsilon>0$ be arbitrary. Then
\begin{align}
\left\|A_{-1}(s)x-A_{-1}(\tilde{s})x\right\|_{-1} =&\left\|(\lambda_0-A_{-1}(s))x-(\lambda_0-A_{-1}(\tilde{s}))x\right\|_{-1} \notag \\
\leq &\left\|(\lambda_0-A_{-1}(s))(x-x_n)\right\|_{-1}+\left\|(\lambda_0-A_{-1}(s))x_n-(\lambda_0-A_{-1}(\tilde{s}))\tilde{x}_n\right\|_{-1} \notag \\
& +\left\|(\lambda_0-A_{-1}(\tilde{s}))(\tilde{x}_n-x)\right\|_{-1} \notag\\
\leq&\left\|x_n-x\right\|+\left\|\lambda_0(x_n-\tilde{x}_n)\right\|_{-1}+\left\|A_{-1}(\tilde{s})\tilde{x}_n-A_{-1}(s)x_n\right\|_{-1}+\left\|\tilde{x}_n-x\right\|,\label{eq:epsilon/4}
\end{align}
where in the last step we used the isometric property of operators $\lambda_0-A_{-1}(s)$ and $\lambda_0-A_{-1}(\tilde{s})$, see Proposition \ref{prop:extra}.
For $n$ large enough, we have $\left\|x_n-x\right\|<\varepsilon/4$ and $\left\|\tilde{x}_n-x\right\|<\varepsilon/4$. 
Further, by applying the definition of both sequences we get
\[
\left\|\lambda_0 (x_n-\tilde{x}_n)\right\|_{-1}\leq\lambda_0 C n\left\|R(n,A(s))x-R(n,A(\tilde{s}))x\right\|
\]
and, by \eqref{eq:res}, also
\[
\left\|A_{-1}(\tilde{s})\tilde{x}_n-A_{-1}(s)x_n\right\|_{-1}\leq n^2 D\left\|R(n,A(s))x-R(n,A(\tilde{s}))x\right\|.
\]
that for some constants $C>0$ and $D>0$. 
As already observed in the proof of Proposition \ref{prop:MultiSemi}, the assumption $c_j(\cdot)\in\mathrm{C}(\RR)$ implies the strong continuity of the map $s\mapsto R(n,A(s))$. In particular, we can find $\delta>0$ such that for $\left|s-\tilde{s}\right|<\delta$ one has that
\[
\left\|R(n,A(s))x-R(n,A(\tilde{s}))x\right\|<\frac{\varepsilon}{4n^2},
\]
showing that also the remaining two summands in \eqref{eq:epsilon/4} become arbitrarily small as soon as $\left|s-\tilde{s}\right|<\delta$. We have thus shown that $\left\|A_{-1}(s)x-A_{-1}(\tilde{s})x\right\|<\varepsilon$ whenever $\left|s-\tilde{s}\right|<\delta$.} {Now, by combining the continuity of $\tilde{s}\mapsto A_{-1}(\tilde{s})x$ with the continuity of $(\tilde{s},t)\mapsto T_{-1,\tilde{s}}(t)f(\tilde{s})$ (shown in  the proof of Proposition \ref{prop:MultiSemi}) we see that the first limit in \eqref{eq:last} equals $A_{-1}(s)f(s)$. The second limit,  by definition, also equals $A_{-1}(s)f(s)$ and hence, the first two limits in  \eqref{eq:last}  cancel out.} We obtain that $\lim\limits_{t\to0}{\frac{f(s-t)-f(s)}{t}}$  exists in $\mathrm{C}_0(\RR,X_{-1})$ and thus $f\in\dom(\mathcal{B})=\mathrm{C}^1_0(\RR,X_{-1})$
which finally yields
\[
\mathrm{C}_0(\RR,X)\cap \dom(\widetilde{\mathcal{G}}) =\mathrm{C}_0(\RR,X)\cap\mathrm{C}^1_0(\RR,X_{-1}).
\]
 \end{proof}
}

Now we are in the position to show that  problem \eqref{nACP-F} - and thus \eqref{eqn:nF} - is well-posed.

\begin{corollary}\label{cor:WellPosnF}\footnote{This is unfortunately false in general, see Erratum on p.~\pageref{erratum}.}
Assume that $c_j(\cdot)\in\mathrm{C}(\RR)$ and that there exist {$m>0$ and $M>0$} such that $m\leq c_j(t)\leq M$ for each $t\in\RR$. Then \eqref{eqn:nF} is a well-posed problem. 
\end{corollary}

\begin{proof}
First observe that, $\dom(\mathcal{G})\cap\mathrm{C}_0^1(\RR,X)=\dom(\mathcal{G})$ as we showed that $\mathrm{C}_0(\RR,X)\cap \dom(\widetilde{\mathcal{G}}) =\mathrm{C}_0(\RR,X)\cap\mathrm{C}^1_0(\RR,X_{-1})$ and $\mathcal{G}$ is the part of $\widetilde{\mathcal{G}}$ in $\mathrm{C}_0(\RR,X)$.
By Theorem \ref{prop:DomEvoSemi}, one has that $\dom(\mathcal{G})$ is the domain of the evolution semigroup generator, so $\dom(\mathcal{G})$ is a core.  Hence, by applying Theorem \ref{thm:WPEvoSemi} we obtain the desired result.
\end{proof}

Unfortunately, we can not give an explicit expression for the corresponding evolution family due to the fact that there is no closed formula for semigroups $(T_s(t))_{t\geq0}$. Theoretically speaking, given an evolution semigroup $(\mathcal{T}(t))_{t\geq0}$ on $\mathrm{C}_0(\RR,X)$ one defines the corresponding evolution family implicitly by
\begin{align}\label{eqn:EvoFamBySemi}
U(t,s)x:=(T_{\mathrm{r}}(s-t)\mathcal{T}(t-s)f)(s),\quad s\in\RR,\ t\geq s,
\end{align}
for $f\in\mathrm{C}_0(\RR,X)$ with $f(s) = x$ and $(T_\mathrm{r}(t))_{t\geq0}$ the right-translation semigroup on $\mathrm{C}_0(\RR,X)$.

\begin{remark}
\begin{abc}
	\item Following the construction of the evolution family given by \eqref{eqn:EvoFamBySemi} via the contraction semigroups $(T_s(t))_{t\geq0}$, $(\mathcal{S}(t))_{t\geq0}$ , $(\mathcal{T}(t))_{t\geq0}$, and translation semigroups we see that $(U(t,s))_{t\geq s}$ actually becomes an evolution family of contractions, i.e., $\left\|U(t,s)\right\|\leq1$ for all $t\geq s$. 
	\item Since each of the operators $(A(t),\dom(A(t)))$, $t\in\RR$,  generates a positive operator semigroup, our construction also yields that  the evolution family $(U(t,s))_{t\geq s}$ obtained in \eqref{eqn:EvoFamBySemi}  consists of positive operators.
 \item  The analysis of the asymptotic behavior of the solutions to non-autonomous problems is in general a very difficult task. 
Batty, Chill, and Tomilov characterized strong stability of a bounded evolution family on a Banach space $X$ by the stability of its associated evolution semigroup on $\mathrm{L}^p( \RR_+,X)$ for some $1 \le p\le \infty$ or, equivalently, by the density of the range of the corresponding  generator on $\mathrm{L}^1( \RR_+,X)$ see  \cite[Thm.~2.2]{BCT2002}. These conditions are however not easy to verify in our case. 

\item Note, that most of our results hold also in the case of non-constant  weights  $w_{ij}(t)$ on the edges as long as  \eqref{eqn:noAbsorb} holds for every $t\in\RR$. Therefore we obtained a generalisation of well-posedness results from Bayazit et al. \cite{BDF2013}. The authors of \cite{BDF2013} also studied longterm  behaviour for the operators with periodic non-autonomous domain. They used an explicit description of the evolution family and the fact that in this case the period-map is a multiplication operator. 
However, such results are not available for our general situation so, a-priori, no statements regarding the asymptotical behaviour are possible.

\end{abc}
\end{remark}

Finally, we consider a slightly modified version of \eqref{eqn:nF} by adding a time-varying absorption term. In particular, we consider the equation given by
\begin{align}\label{eqn:pnF}
\tag{pnF}
\begin{cases}
\frac{\partial}{\partial t}u_j(t,x)=c_j(t)\frac{\partial}{\partial x}u_j(t,x)+q_j(t,x)u_j(t,x),&\quad x\in\left(0,1\right),\ t\geq 0,\\
u_j(t,0)= f_j(x),&\quad x\in\left(0,1\right),\\ 
\phi^-_{ij}c_j(t)u_j(t,1)=\omega_{ij}\sum_{k\in\mathbb{N}}{\phi^+_{ik}c_k(t)u_k(t,1)},&\quad t\geq0.
\end{cases}
\end{align}
Then \eqref{eqn:pnF} can be written as an abstract Cauchy problem of the form
\begin{align}\label{eqn:nACPp}
\begin{cases}
\dot{u}(t)=(A(t)+B(t))u(t),&\quad t,s\in\RR,\ t\geq s,\\
u(s)=\mathsf{f},
\end{cases}
\end{align}
where the family of operators $(A(t),\dom(A(t)))_{t\in\RR}$ is again defined by \eqref{eqn:DiffOp}--\eqref{eqn:DiffOp-D}, $\mathsf{f} = (f_j)_{j=1}^m$,  and the family of bounded linear operators $(B(t))_{t\in\RR}$ is defined by
\begin{align}
B(t):=\mathrm{diag}(M_{{ q_j(t,\cdot)}}),
\end{align}
where $M_{{ q_j(t,\cdot)}}$ denotes the multiplication operator { on $X$ induced by the function $q_j(t,\cdot)$. Under sufficient regularity conditions on functions $q_j$, classical bounded perturbation results \cite[Cor.~VI.9.20]{EN} yield the following.}

\begin{corollary}\label{thm:PertNonAutFlow}\footnote{This is unfortunately false in general, see Erratum on p.~\pageref{erratum}.}
Let $c_j(\cdot)\in\mathrm{C}(\RR)$ such that there exist {$m>0$ and $M>0$} with $m\leq c_j(t)\leq M$ for each $t\in\RR$. Under the assumption that $q_j(t,\cdot)\in\mathrm{L}^{\infty}\left(\left[0,1\right],\CC^m\right)$ and $q_j(\cdot,x)\in\mathrm{C}_\mathrm{b}(\RR,X)$, there exists a unique  evolution family $(U_B(t,s))_{t\geq s}$, generated by the family of operators $(A(t)+B(t),\dom(A(t)))_{t\in\RR}$.
\end{corollary}

It is worth to mention, that Corollary \ref{thm:PertNonAutFlow} does not imply that \eqref{eqn:pnF} is well-posed, see  \cite[Ex.~VI.9.21]{EN}. Nonetheless, for the evolution family  $(U_B(t,s))_{t\geq s}$ obtained in the theorem, the function  $U_B(\cdot,s)\mathsf{f}$ can be interpreted as a \emph{mild solution} of the problem \eqref {eqn:nACPp}.

\section*{Data Availability Statements}
Data sharing not applicable to this article as no datasets were generated or analysed during the current study.

\section*{Acknowledgement}
{The authors are indebted to Abdelaziz Rhandi for valuable discussions and his contribution to the proof of Theorem \ref{prop:DomEvoSemi}. They are also thankful to the anonymous referee for the valuable feedback on the manuscript.} Both authors acknowledge financial support,  C.B. from the Deutsche Forschungsgemeinschaft (DFG, German Research Foundation) - 468736785 and M.K.F from the Slovenian Research Agency, Grant No.~P1-0222. {This work is also based on the research supported by the National Research Foundation of South Africa (Grant number: SRUG220317136). It is acknowledged that opinions, findings and conclusions or recommendations expressed in any publication generated by this supported research is that of the author(s). The National Research Foundation accepts no liability whatsoever in this regard.} Moreover, this article is based upon work from COST Action 18232 MAT-DYN-NET, supported by COST (European Cooperation in Science and Technology), www.cost.eu.

\newpage

\begin{center}
\textbf{ERRATUM ON: WELL-POSEDNESS OF NOM-AUTONOMOUS TRANSPORT EQUATION ON METRIC GRAPHS}
\end{center}

\bigskip
\begin{center}
CHRISTIAN BUDDE AND MARJETA KRAMAR FIJAV\v{Z}
\end{center}

\bigskip
This erratum concerns the publication \cite{BKF24}. The authors observed that the family of operators defined in (4.2) does not define an operator semigroup as claimed. This is due to the fact that the multiplication semigroup $(\mathcal{S}(t))_{t\geq0}$ and the translation semigroup do not commute unless one deals with the autonomous situation as it is done in \cite{NNR1996} or one has the constant domain situation as studied in Section 3.2. Hence, Theorem~4.4 and its Corollaries do not hold as stated. 
\label{erratum}

Moreover, let us add that also  the equality $\dom(\mathscr{G})\cap\mathrm{C}_0^1(\mathbb{R},X)=\dom(\mathscr{G})$ stated at the beginning of the proof of Corollary 4.5. does not hold in general, as assumed. 

\section*{Acknowledgement}
We would like to thank Roland Schnaubelt, Sascha Trostorff, and Marcus Waurick for bringing this issues to our attention.

\end{document}